\documentclass[11pt,a4paper,leqno]{article}
\usepackage[nocolor]{mymacros}
\usepackage{cite}

\title{Log-Sobolev inequality for near critical Ising models}

\author{Roland Bauerschmidt\footnote{University of Cambridge, Statistical Laboratory, DPMMS. E-mail: {\tt rb812@cam.ac.uk}.}
\and
Benoit Dagallier\footnote{University of Cambridge, Statistical Laboratory, DPMMS. E-mail: {\tt bd444@cam.ac.uk}.}}

\date{\vspace*{-2em}} 

\begin{document}
\maketitle
\begin{abstract}
For general ferromagnetic Ising models whose coupling matrix has bounded spectral radius,
we show that the log-Sobolev constant satisfies a simple bound expressed only
in terms of the susceptibility of the model.
This bound implies very generally that the log-Sobolev constant is uniform in the
system size up to the critical point (including on lattices), without using any mixing conditions.
Moreover, if the susceptibility satisfies the mean-field bound as the
critical point is approached, our bound implies that the log-Sobolev constant
depends polynomially on the distance to the critical point and on the volume. In particular, this
applies to the Ising model on subsets of $\mathbb{Z}^d$ when $d>4$.

The proof uses a general criterion for the log-Sobolev inequality in terms of the
Polchinski (renormalisation group) equation, a recently proved remarkable
correlation inequality for Ising models with general external fields, the
Perron--Frobenius theorem, and the log-Sobolev inequality for product Bernoulli
measures.
\end{abstract}

\section{Main result and proof}

\subsection{Main result}

Let $\Lambda$ be a finite set, and denote by $(\cdot,\cdot)$ the standard inner product on $\R^\Lambda$.
For $A= (A_{xy})_{x,y\in\Lambda}$ a symmetric matrix with $A_{xy}\leq 0$ if $x \neq y$, 
the ferromagnetic Ising model with coupling matrix $A$, inverse temperature $\beta\geq 0$, and external field $h\in\R^\Lambda$ has the expectation
\begin{equation} \label{e:def-Ising}
  \E_{\mu_{\beta,h}} F
  \propto
  \sum_{\sigma \in \{\pm 1\}^\Lambda}e^{-\frac{\beta}{2} (\sigma, A\sigma)} e^{(h,\sigma)} F(\sigma).
\end{equation}
We assume that $A$ is positive definite and has spectral radius $\|A\| \leq 1$.
The positive definiteness can always be imposed by adding a diagonal matrix (without changing the corresponding Ising model)
and $\|A\|\leq 1$ corresponds to a normalisation of the inverse temperature parameter.

For $x\in\Lambda$, 
let $\sigma^x \in \{\pm 1\}^\Lambda$ be the spin configuration obtained from $\sigma \in \{\pm 1\}^{\Lambda}$ by flipping the sign of $\sigma_x$.
The standard (Glauber) Dirichlet form associated with the Ising model \eqref{e:def-Ising} is
\begin{equation} \label{e:Dirichlet}
  D_{\mu_{\beta,h}} (F)  = \frac12 \sum_{x\in\Lambda}\E_{\mu_{\beta,h}} \Big[\big(F(\sigma)-F(\sigma^x)\big)^2\Big].
\end{equation}
The log-Sobolev constant $\gamma_{\beta,h}$ is the largest constant such that, for all nonnegative $F$,
\begin{equation}
  \ent_{\mu_{\beta,h}}(F) \leq \frac{2}{\gamma_{\beta,h}} D_{\mu_{\beta,h}}(\sqrt{F}),
\end{equation}
where $\ent_{\mu_{\beta,h}}(F) = \E_{\mu_{\beta,h}}\Phi(F)-\Phi(\E_{\mu_{\beta,h}} F)$ with $\Phi(x)=x\log x$
is the relative entropy.
Our main result is stated in terms of the $0$-field susceptibility of the Ising model, given by
\begin{equation} \label{e:suscept-def}
  \chi_\beta = \sup_{x\in\Lambda} \sum_{y\in\Lambda} \E_{\beta,0}(\sigma_x\sigma_y).
\end{equation}

\begin{theorem} \label{thm:ising}
  The log-Sobolev constant of \eqref{e:def-Ising} satisfies
  \begin{equation} \label{e:LSbound}
    \frac{1}{\gamma_{\beta,h}} \leq \frac{1}{2}+\int_0^\beta e^{2\int_0^t \chi_s\, ds}\, dt.
  \end{equation}
\end{theorem}

The Dirichlet form is associated with a Markov process (Glauber dynamics), see \cite{MR1746301} for background.
The standard Dirichlet form \eqref{e:Dirichlet} corresponds to a convenient choice of transition rates,
but other common choices of rates (such as heat-bath and\ Metropolis dynamics) are equivalent when the rates are bounded, as is the
case for the standard Ising model on $\Lambda \subset \Z^d$, also see \cite{MR1746301}.
The log-Sobolev inequality can be equivalently formulated in terms of hypercontractivity of the Glauber semigroup
and has many well-known implications.
In particular, the inverse log-Sobolev constant bounds the relaxation rate to equilibrium in the sense of relative entropy, i.e.,
if $F_t\, \mu_{\beta,\sigma}$ denotes the law of the dynamics at time $t\geq 0$, then
  \begin{equation}
  \ent_{\mu_{\beta,h}}(F_t)
  \leq
  e^{-2\gamma_{\beta,h}t} \ent_{\mu_{\beta,h}}(F_0)
  .
  \label{e:relaxation}
  \end{equation}
The inverse log-Sobolev constant can also be used to obtain bounds on other notions of convergence to equilibrium in finite and infinite volume, 
see the references
\cite{MR1746301,MR1971582,MR1490046,MR3020173} for introductions to the general implications.

Since $\chi_s$ is increasing in $s$ (by the second Griffiths inequality),
Theorem~\ref{thm:ising} shows in particular that
whenever $\chi_\beta$ is bounded,
the log-Sobolev constant is bounded below:
\begin{equation}
  \frac{1}{\gamma_{\beta,h}} \leq \frac12+ \beta e^{2 \beta \chi_\beta}.
\end{equation}
This shows that the log-Sobolev inequality holds uniformly in the volume in the entire high temperature phase of the Ising model,
e.g., on $\Lambda \uparrow \Z^d$ but equally in a much more general setting. 
For lattices $\Lambda$, uniformity in the volume was recently proved in \cite{2107.09243} by establishing the strong spatial mixing property, see
\cite{MR1746301}. Our more general criterion does not rely on geometry or mixing conditions.
The new correlation inequality from \cite{2107.09243} is also important in our proof.

But even more interestingly, the bound \eqref{e:LSbound} shows that if the susceptibility satisfies the mean-field bound
then the log-Sobolev constant is polynomial (in the distance to the critical point and in the volume).
This is summarised in the following corollary, whose proof simply consists of carrying out the integrals in \eqref{e:LSbound}
using the assumed bound on $\chi_\beta$.

\begin{corollary}
  Let $D>1/2$ and assume that the mean-field bound $\chi_\beta \leq D/(\beta_c-\beta)$ for all $\beta<\beta_c$ holds. Then
  \begin{equation}
    \frac{1}{\gamma_{\beta,h}} \leq
    \frac12+ \frac{\beta_c}{2D-1} \big[(1-\beta/\beta_c)^{1-2D} -1\big]
    \underset{\beta\rightarrow\beta_c}{\sim} \frac{\beta_c}{2D-1} (1-\beta/\beta_c)^{1-2D}.
  \end{equation}
  Similarly, if the finite-volume mean-field bound $\chi_\beta \leq D/(\beta_c-\beta+L^{-2})$ for all $\beta<\beta_c$ holds, then
 \begin{equation}
  \frac{1}{\gamma_{\beta_c,h}} \leq
  \frac12 + \frac{\beta_c+L^{-2}}{2D-1} \big[(L^2\beta_c+1)^{2D-1}-1\big]
  \underset{L \to\infty}{\sim}
  \frac{\beta_c}{2D-1} (L^2\beta_c+1)^{2D-1}.
\end{equation}
\end{corollary}


The mean-field bound holds for Ising models on $\Z^d$ in $d\geq 5$, see \cite{MR678000,MR720344}, and
the finite-volume mean-field bound similarly holds for the Ising model on a hypercube of side length $L$
(with free boundary conditions) in $d\geq 5$, and with possibly a different exponent for other boundary conditions, see \cite{MR4341075} for discussion of this. 
The exponents above are not the expected optimal ones when $\Lambda \uparrow \Z^d$, which would be $-1$ respectively $+1$ in $d\geq 5$.
We remark though that in the actual mean-field model (the Curie--Weiss model), the mean-field bound holds
with $D=1$ (see, e.g., \cite[Chapter~1]{MR3969983}) and one can thus obtain the optimal exponent from \eqref{e:LSbound}.
However, for the Curie--Weiss model,
also many other methods give sharp results, see, e.g., \cite{MR1182416,MR1746301},
and the more recent spectral conditions \cite{MR3926125,MR4408509,2106.04105}.

For lattices, by contrast,
a polynomial bound on the spectral gap for the Glauber dynamics of the Ising model at the critical temperature has previously only been proved
in two dimensions \cite{MR2945623}.
The proof of this bound relies on the RSW estimates \cite{MR2839298} as input.
Similar estimates have also been obtained for two-dimensional $q$-state Potts models with $q<4$, see \cite{MR3794520}.
The approach of the critical point for these models appears to remain open.
Other polynomial bounds for near critical spin models were restricted to simpler geometries
such as trees \cite{MR2585995} or hierarchical lattices \cite{MR4061408}.

For completeness, we also comment briefly on the low-temperature phase,
in which our technique to prove Theorem~\ref{thm:ising} does not seem directly useful.
In this region, the Glauber dynamics equilibrates at a rate that typically vanishes with the size of $\Lambda$ (assuming the external field $h$ is $0$). 
For the ferromagnetic Ising model in a hypercube $\Lambda \subset\Z^d$ of side length $L$, 
boundary conditions play a crucial role, 
and their influence is not captured by the zero-field susceptibility~\eqref{e:suscept-def}
(zero field corresponds to free boundary conditions). 
For instance, for free boundary conditions, it is known that the spectral gap 
(which is always larger than the log-Sobolev constant $\gamma_{\beta,h}$) 
decays like $e^{-c(\beta)L^{d-1}(1+o_L(1))}$ when $\beta>\beta_c$;
the constant $c(\beta)$ is explicit in dimension two~\cite{MR1746301} and there are more general bounds in dimension three and higher~\cite{zbMATH02123135}. 
Conversely, if the boundary condition is all $+$, there is a long-standing conjecture, 
the so-called Lifshitz law, 
that convergence to equilibrium from a worst-case initial condition should typically take a time of order $L^2$, 
up to logarithmic corrections.
Arguments supporting that bound are given in~\cite{MR1927919}, but 
so far the bound has only been established at $0$ temperature in all dimensions $d\geq 2$, 
see~\cite{zbMATH05906455} in dimension $d=2,3$ and~\cite{zbMATH06145994} in higher dimension. 
For $\beta\in(\beta_c,\infty)$, the currently best known results are in dimension two, 
where the rate of relaxation is known to be going to $0$ at least as $L^{-c'(\beta)\log L}$ for some $c'(\beta)>0$, 
see~\cite{MR3017041} and references therein. 
Similar bounds can also be obtained with periodic boundary conditions instead of all $+$, 
but for well chosen initialisations of the dynamics:  
bounds on the mixing time of order $L^{-1}$ in dimension $d=2$ and $L^{-c\log ^{d-1}L}$ have recently been established in~\cite{2106.11296}.

In \cite{lsiphi4}, we derive a variant of Theorem~\ref{thm:ising} for $\varphi^4$ models and use it to prove that the
continuum $\varphi^4_2$ and $\varphi^4_3$ measures satisfy log-Sobolev inequalities uniformly in the regularisation (and volume in the uniqueness phase).

\begin{remark}
  The proof only uses that $\chi_\beta$ defined in \eqref{e:suscept-def}
  provides an upper bound on the spectral radius of
  the two-point function $(\E_{\mu_{\beta,0}}(\sigma_x\sigma_y))_{x,y\in\Lambda}$.
  Thus replacing $\chi_\beta$  by the spectral radius of the two-point function gives a slightly more general statement.
\end{remark}

\subsection{Proof of Theorem~\ref{thm:ising}}

The strategy of the proof is to decompose $\mu_{\beta,h}$ into two measures: 
an infinite-temperature Ising part, for which the log-Sobolev inequality is known; 
and a continuous part for which the log-Sobolev inequality is established via the criterion 
based on the Polchinski (renormalisation group) equation from \cite[Theorem 2.5]{MR4303014}.
The first part is similar to the one in \cite{MR3926125} where the same kind of decomposition was performed.
In \cite{MR3926125}, the second part was treated
with the Bakry--Emery criterion, leading to the condition $\beta < 1$ (but without assuming the coupling matrix $A$ is ferromagnetic).
In contrast, our use of the more powerful Polchinski equation criterion does not require an upper bound on $\beta$.
To verify the assumptions of the criterion we exploit that $A$ is ferromagnetic by using the FKG inequality and the new correlation inequality from \cite{2107.09243}.

\begin{proof}[Proof of Theorem~\ref{thm:ising}]
Fix $\alpha>\beta$ and define for $t\in [0,\beta]$:
\begin{equation}
  C_t = (tA+(\alpha-t))^{-1}.
\end{equation}
The covariances $C_t$ are positive definite, increasing as quadratic forms in $t\in [0,\beta]$, and satisfy
\begin{equation}
  \dot C_t = (1-A)C_t^2 ,
  \qquad
  \ddot C_t 
  = 2(1-A)C_t\dot C_t.
\end{equation}
The Ising measure $\mu_{\beta,h}$ can then be written as (the term $\alpha-t=\alpha-\beta$ is irrelevant since $\sigma_x^2=1$)
\begin{equation}
  \E_{\mu_{\beta,h}}F \propto
  \sum_{\sigma \in \{\pm 1\}^\Lambda}e^{-\frac12 (\sigma, C_\beta^{-1}\sigma)} e^{(h,\sigma)}F(\sigma)
  \propto
  \sum_{\sigma \in \{\pm 1\}^\Lambda}e^{-\frac{\beta}{2} (\sigma, A\sigma)} e^{(h,\sigma)} F(\sigma).
\end{equation}
For $t\in [0,\beta)$ define the renormalised potential for $\varphi \in \R^\Lambda$ by
\begin{align} \label{e:Vt-Ising}
  V_t(\varphi)
  &= -\log \sum_{\sigma\in \{\pm 1\}^\Lambda} e^{-\frac12 (\sigma-\varphi,C_t^{-1}(\sigma-\varphi))} e^{(h,\sigma)}
    \nnb
  &= \frac12 (\varphi,C_t^{-1}\varphi) -\log \sum_{\sigma\in\{\pm 1\}^\Lambda} e^{-\frac12 (\sigma,C_t^{-1}\sigma)} e^{(h+C_t^{-1}\varphi,\sigma)},
\end{align}
and the renormalised measure as the probability measure on $\R^\Lambda$ given by
\begin{equation} \label{e:nut-Ising}
  \nu_{t,\beta}(d\varphi) \propto e^{-\frac12 (\varphi,(C_\beta - C_t)^{-1}\varphi)} e^{-V_t(\varphi)} \, d\varphi.
\end{equation}
Both $V_t$ and $\nu_{t,\beta}$ also depend on $h$ (not denoted explicitly).
For any $t\in(0,\beta]$ then
\begin{equation}
  e^{-V_t(\varphi)} \propto \int_{\R^\Lambda} e^{-\frac12 (\zeta,(C_t-C_0)^{-1}\zeta)} e^{-V_0(\varphi+\zeta)} \, d\zeta,
\end{equation}
and the Ising measure $\mu_{\beta,h}$ can be decomposed as follows for any $t\in [0,\beta)$ 
(thus including $t=0$):
\begin{equation} \label{e:measure-decomp}
  \E_{\mu_{\beta,h}} F = \E_{\nu_{t,\beta}} \E_{\mu_{t,h+C_t^{-1}\varphi}} F.
\end{equation}
The last two equations follow from the following convolution identity for Gaussian densities, 
valid for all $0 \leq t <s \leq \beta$ and $\sigma \in \R^\Lambda$:
\begin{equation} \label{e:convolv}
  e^{-\frac12 (\sigma, C_s^{-1}\sigma)}
  \propto
  \int_{\R^\Lambda} e^{-\frac12 (\varphi,(C_s-C_t)^{-1}\varphi)} e^{-\frac12 ((\sigma-\varphi),C_t^{-1}(\sigma-\varphi))} \, d\varphi.
\end{equation}

Let $F:\{-1,1\}^{\Lambda}\rightarrow \R_+$ be fixed. 
The relative entropy can be decomposed using \eqref{e:measure-decomp}:
\begin{equation} \label{e:ent-decomp}
  \ent_{\mu_{\beta,h}}(F) = \E_{\nu_{t,\beta}} (\ent_{\mu_{t,h+C_t^{-1}\varphi}} (F(\sigma))) + \ent_{\nu_{t,\beta}}(G(\varphi)),
\end{equation}
where
\begin{equation}
  G(\varphi) = \E_{\mu_{t,h+C_t^{-1}\varphi}} F(\sigma).
\end{equation}
In fact, we only use $t=0$ in this identity.
Then $C_0^{-1}=\alpha$ and the measure $\mu_{0,h+C_0^{-1}\varphi}$ is the infinite temperature (product) Ising model with external field $h+\alpha\varphi$. 
It thus satisfies a log-Sobolev inequality uniformly in $h$ and $\varphi$ with log-Sobolev constant $2/\gamma=1$,
see \cite{MR1849347,MR1490046,MR1845806}.
Therefore
\begin{equation}
  \E_{\nu_{0,\beta}} (\ent_{\mu_{0,h+\alpha\varphi}}(F(\sigma)))
  \leq
  \E_{\nu_{0,\beta}}D_{\mu_{0,h+\alpha\varphi}}(\sqrt{F})
  =
  D_{\mu_{\beta,h}}(\sqrt{F}).\label{eq_LSI_product}
\end{equation}
For the second term in \eqref{e:ent-decomp}, we show that
$\nu_{0,\beta}$ satisfies a log-Sobolev inequality (with constant denoted by $\gamma$) by using the Polchinski equation criterion from \cite[Theorem~2.5]{MR4303014}. 
This criterion states that, for a family $(\dot\lambda_t)_{t\in[0,\beta]}$ of real numbers
and increasing covariances $C_t$ as above:
\begin{equation} 
\forall t\in [0,\beta],
\qquad
\dot C_t\He V_t(\varphi)\dot C_t -\frac12 \ddot C_t\geq \dot\lambda_t \dot C_t
\end{equation}
implies
\begin{equation} \label{e:criterion}
  \frac{1}{\gamma} \leq \|\dot C_0\| \int_0^\beta e^{-2\int_0^t\dot \lambda_s \, ds}\, dt.
\end{equation}
To be precise, in \cite[Theorem~2.5]{MR4303014}, the covariances $C_t$ are parametrised by $t\in [0,\infty]$ rather than $t\in [0,\beta]$, and start at $C_0=0$.
This difference simply corresponds to the change of variable from $t$ to $\beta(1-e^{-t})$ in the definition of $C_t$, 
and to replacing $C_t$ by $C_t-C_0$.
We prefer the parametrisation by $t\in [0,\beta]$ here due to its more transparent interpretation as a change in inverse temperature.

To find a family $(\dot\lambda_t)_{t\in[0,\beta]}$, note that \eqref{e:Vt-Ising} implies
\begin{equation}
  \He V_t(\varphi) = C_t^{-1} - C_t^{-1}\Sigma_t(h+C_t^{-1}\varphi) C_t^{-1} 
\end{equation}
where
\begin{equation}
  \Sigma_t(f) = \pB{\E_{\mu_{t,f}}(\sigma_x\sigma_y)-\E_{\mu_{t,f}}(\sigma_x)\E_{\mu_{t,f}}(\sigma_y)}_{x,y},\qquad f \in \R^\Lambda.
\end{equation}
Since $\mu_{t,f}$ is a ferromagnetic Ising model (at inverse temperature $t$ and with external field $f\in\R^\Lambda$),
the FKG inequality implies that $\Sigma_t(f)$ has nonnegative entries, and
the recently proved powerful correlation inequality \cite[Corollary 1.3]{2107.09243}
further shows that $\Sigma_t(f)_{xy} \leq \Sigma_t(0)_{xy}$ for all $x,y\in\Lambda$.
Finally, by the Perron--Frobenius theorem, there exists $Y_0 = Y_0(f) \in \R^\Lambda$ with nonnegative entries and $\|Y_0\|_2=1$ such that
\begin{align}
  \|\Sigma_t(f)\| &= Y_0^T\Sigma_t(f)Y_0 \nnb
  &\leq Y_0^T\Sigma_t(0)Y_0 \leq \|\Sigma_t(0)\| 
  \leq \sup_{x\in\Lambda} \sum_{y\in\Lambda} \Sigma_t(0)_{xy}
  = \chi_t,
\end{align}
where $\|\Sigma\| = \sup_{\|Y\|_2\leq 1} \|\Sigma Y\|_2$  is the spectral radius of the symmetric matrix $\Sigma$,
and the last inequality follows from the Cauchy-Schwarz inequality.
Thus for all $X \in \R^\Lambda$, with $f=h+C_t^{-1}\varphi$,
\begin{equation}
  X^TC_t^{-1}\Sigma_t(h+C_t^{-1}\varphi)C_t^{-1}X \leq \|C_t^{-1}X\|_2^2 \|\Sigma_t(h+C_t^{-1}\varphi)\| \leq \chi_t X^TC_t^{-2}X,
\end{equation}
and we have the quadratic form inequality ($C_t,\dot C_t,\ddot C_t$, and $A$ all commute)
\begin{equation}
  \dot C_t\He V_t(\varphi)\dot C_t -\frac12 \ddot C_t\geq
  \dot C_t^2 C_t^{-1} - \chi_t \dot C_t^2 C_t^{-2} - (1-A) C_t \dot C_t
  = -(1-A) \dot C_t \chi_t \geq -\chi_t\dot C_t.
\end{equation}
Therefore by \eqref{e:criterion},
\begin{equation}
  \ent_{\nu_{0,\beta}}(G)
  \leq \frac{2}{\gamma} \E_{\nu_{0,\beta}}(\nabla \sqrt{G})^2\label{eq_LSI_with_G}
\end{equation}
with
\begin{equation} \label{e:gamma0beta}
  \frac{1}{\gamma}
  \leq \|\dot C_0\| \int_{0}^{\beta} e^{2\int_0^t\chi_s\, ds} \, dt
  = \frac{1}{\alpha^2} \int_{0}^{\beta} e^{2\int_0^t\chi_s\, ds} \, dt.
\end{equation}
To conclude, it only remains to bound $(\nabla_{\varphi}\sqrt{G})^2$ in terms of $F$.
This argument is exactly as in \cite{MR3926125}.
Indeed, by the definitions of $G$ (with $t=0$) and of $\mu_{0,h+\alpha\varphi}$,
\begin{equation}
\forall x\in\Lambda,\qquad  
(\nabla_{\varphi_x}\sqrt{G})^2 = \alpha^2 \frac{\cov_{\mu_{0,h+\alpha\varphi}}(F, \sigma_x)^2}{4(\E_{\mu_{0,h+\alpha\varphi}}F)}.\label{eq_nabla_G}
\end{equation}
Denote by $\mu_{0,h+\alpha\varphi}^x$ the conditional measure of the product measure $\mu_{0,h+\alpha\varphi}$ with all $\sigma_y$ with $y\neq x$ fixed.
By duplication, the covariance can be written as
\begin{align}
  \cov_{\mu^x_{0,h+\alpha\varphi}}(F, \sigma_x)^2
  &= \pa{ \frac12 \E_{\mu_{0,h+\alpha\varphi}^x \otimes \mu_{0,h+\alpha\varphi}^x} \qa{ (\sqrt{F}-\sqrt{F'}) (\sqrt{F}+\sqrt{F'})(\sigma_x-\sigma_x')}}^2
    \nnb
  &\leq \var_{\mu_{0,h+\alpha\varphi}^x}(\sqrt{F}) \frac12 \E_{\mu_{0,h+\alpha\varphi}^x \otimes \mu_{0,h+\alpha\varphi}^x} \qa{ (\sqrt{F}+\sqrt{F'})^2(\sigma_x-\sigma_x')^2}
    \nnb
  &\leq 8 \var_{\mu_{0,h+\alpha\varphi}^x}(\sqrt{F}) \E_{\mu_{0,h+\alpha\varphi}^x} (F),
\end{align}
where we have used the Cauchy-Schwarz inequality and that $|\sigma_x|\leq 1$.
Since $\mu_{0,h+\alpha\varphi}$ is a product measure, $\cov_{\mu_{0,h+\alpha\varphi}}(F,\sigma_x) = \E_{\mu_{0,h+\alpha\varphi}}\cov_{\mu_{0,h+\alpha\varphi}^x}(F,\sigma_x)$.
Thus, using the Cauchy-Schwarz inequality on the third line:
\begin{align}
\cov_{\mu_{0,h+\alpha\varphi}}(F,\sigma_x)^2 
&= 
\E_{\mu_{0,h+\alpha\varphi}}\big[\cov_{\mu_{0,h+\alpha\varphi}^x}(F,\sigma_x)\big]^2  \nnb
&\leq 
8\E_{\mu_{0,h+\alpha\varphi}}\Big[\sqrt{\var_{\mu_{0,h+\alpha\varphi}^x}(\sqrt{F})}\sqrt{\E_{\mu_{0,h+\alpha\varphi}^x} (F)}\Big]^2\nnb
&\leq
8\E_{\mu_{0,h+\alpha\varphi}}\Big[\var_{\mu_{0,h+\alpha\varphi}^x}(\sqrt{F})\Big]\E_{\mu_{0,h+\alpha\varphi}}(F).
\end{align}
Injecting the last bound into \eqref{eq_nabla_G} and using the elementary fact that the
biased Bernoulli $\pm 1$ measure $\mu_{0,h+\alpha\varphi}^x$ has spectral gap at least $2$ (with our normalisation of Dirichlet form),
\begin{align}
  \E_{\nu_{0,\beta}}(\nabla_{\varphi_x}\sqrt{G})^2
  &\leq 2\alpha^2   \E_{\nu_{0,\beta}} \E_{\mu_{0,h+\alpha\varphi}}\Big[\var_{\mu_{0,h+\alpha\varphi}^x}(\sqrt{F})\Big]
    \nnb
  &\leq 
  2\alpha^2
  \E_{\nu_{0,\beta}} \E_{\mu_{0,h+\alpha\varphi}}
  \frac14 \E_{\mu_{0,h+\alpha\varphi}^x} \Big[\big(\sqrt{F(\sigma)}-\sqrt{F(\sigma^x)}\big)^2\Big]
  .
\end{align}
Taking the sum over $x$ and using that $\E_{\mu_{0,h+\alpha\varphi}} \E_{\mu_{0,h+\alpha\varphi}^x} = \E_{\mu_{0,h+\alpha\varphi}}$, therefore
\begin{equation} \label{e:FGbd}
  \E_{\nu_{0,\beta}}(\nabla_{\varphi}\sqrt{G})^2
  \leq \alpha^2 D_{\mu_{\beta,h}}(\sqrt{F}).
\end{equation}
In summary, putting together the entropy decomposition \eqref{e:ent-decomp} and the bounds \eqref{eq_LSI_product} and \eqref{eq_LSI_with_G}--\eqref{e:FGbd} on each term of this decomposition:
\begin{equation}
  \ent_{\mu_{\beta,h}} (F)
  \leq (1+\frac{2\alpha^2}{\gamma}) D_{\mu_{\beta,h}}(\sqrt{F}).
\end{equation}
Substituting \eqref{e:gamma0beta} into this bound gives the result.
\end{proof}

\section*{Acknowledgements}

We thank T.\ Bodineau, J.\ Ding, and T.\ Helmuth for helpful discussions.
This work was supported by the European Research Council under the European Union's Horizon 2020 research and innovation programme
(grant agreement No.~851682 SPINRG). R.B.\ also acknowledges the hospitality of the Department of Mathematics at McGill University
where part of this work was carried out.

\bibliography{all}
\bibliographystyle{plain}
\end{document}